\documentclass[12pt]{amsart}

\usepackage{amssymb, amscd, txfonts}
\usepackage{graphicx}

\numberwithin{equation}{section}

\newtheorem{thm}{Theorem}[section]
\newtheorem{prop}[thm]{Proposition}
\newtheorem{lem}[thm]{Lemma}

\theoremstyle{definition}
\newtheorem{defn}[thm]{Definition}

\theoremstyle{remark}
\newtheorem{rem}[thm]{Remark}

\renewcommand{\hom}{\operatorname{Hom}}

\newcommand{\Z}{\mathbb{Z}}

\newcommand{\R}{\mathbb{R}}
\newcommand{\C}{\mathbb{C}}
\newcommand{\F}{\mathbb{F}}

\DeclareMathOperator{\im}{Im}

\begin{document}

\title[Reidemeister torsion and Seifert surgery]{Reidemeister torsion for linear representations and Seifert surgery on knots}
\author[T.~Kitayama]{Takahiro KITAYAMA}
\address{Graduate~School~of~Mathematical~Sciences, the~University~of~Tokyo, 3-8-1~Komaba, Meguro-ku, Tokyo 153-8914, Japan}
\email{kitayama@ms.u-tokyo.ac.jp}
\subjclass[2000]{Primary~57M27, Secondary~55R55, 57Q10}
\keywords{Reidemeister torsion, Dehn surgery, Seifert fibered space}

\begin{abstract}
We study an invariant of a $3$-manifold which consists of Reidemeister torsion for linear representations which pass through a finite group.
We show a Dehn surgery formula on this invariant and compute that of a Seifert manifold over $S^2$.
As a consequence we obtain a necessary condition for a result of Dehn surgery along a knot to be Seifert fibered, which can be applied even in a case where abelian Reidemeister torsion gives no information.
\end{abstract}

\maketitle

\section{Introduction}
Let $K$ be a knot in a homology $3$-sphere and $E_K$ the complement of an open tubular neighborhood of $K$.
We denote by $K(p / q)$ the result of $p / q$-surgery along $K$ for an irreducible fraction $p / q$.
The aim of the paper is to give a necessary condition for $K(p / q)$ to be a certain closed $3$-manifold, in particular a Seifert manifold, using Reidemeister torsion for linear representations.

It is known that the Alexander polynomial $\Delta_K$ of $K$ has useful information on Dehn surgery.
In \cite{Ka1} and \cite{Ka2} Kadokami used abelian Reidemeister torsion to provide obstructions to lens surgery and Seifert surgery in terms of $\Delta_K$.
In \cite{OS1}, \cite{OS2} and \cite{KMOS} Ozsv\'{a}th-Szab\'{o} and Kronheimer-Mrowka-Ozsv\'{a}th-Szab\'{o} gave other obstructions for $K \subset S^3$ to lens surgery and Seifert surgery in terms of the Heegaard Floer homology of $K(0)$, the knot Floer homology of $K$ and the Monopole Floer homology of $K(0)$, which deduce those in terms of $\Delta_K$.
It is of interest to investigate information on Dehn surgery that Reidemeister torsion for linear representations has.
Reidemeister torsion of $E_K$ coincides with a twisted Alexander invariant of $K$ up to multiplication of units.
See \cite{KL}, \cite{Ki1}, \cite{L} and \cite{W} for the definition of twisted Alexander invariants and the relation with Reidemeister torsion.

We fix orientations of $K$ and the ambient homology sphere.
Let $M$ be a closed connected $3$-manifold with $H_1(M) = \Z / p$ and $\varphi \colon G \to GL_n(\F)$ a linear representation over a field $\F$ of a finite group $G$.
All homology groups and cohomology groups are with respect to integral coefficients unless specifically noted.
First we define an invariant $T_K^{\varphi}([g, h])$ of $K$ for $[g, h] \in G \times G / G$, where $G$ acts on $G \times G$ by
\[ g' \cdot (g, h) := (g' g g'^{-1}, g' h g'^{-1}) \]
for $g' \in G$ and $(g, h) \in G \times G$, and an invariant $T_{M, \beta}^{\varphi}$ of $M$ for a surjection $\beta \colon \pi_1 M \to \langle \zeta \rangle$, where $\zeta \in \overline{\F}$ is a primitive $p$-root of $1$ (Definition \ref{def_T}).
These invariants are sets which consist of Reidemeister torsion of $E_K$ and $M$ respectively for representations which pass through $G$ surjectively.
The pair $[g, h]$ corresponds with the images of longitudinal and meridional elements by the representations.
It is worth pointing out that for $K \subset S^3$, if we know all surjective homomorphisms from $\pi_1 E_K$ to $G$, $T_K^{\varphi}([g, h])$ is combinatorially computable from a presentation of $\pi_1 E_K$ as Reidemeister torsion is.
We establish a Dehn surgery formula which computes $T_{K(p / q), \beta}^{\varphi}$ from $T_K^{\varphi}([g, h])$ with $g^q h^p = 1$ (Theorem \ref{thm_S}).
Therefore by this formula we obtain a necessary condition for $K(p / q)$ to be homeomorphic to $M$ if we have $T_{M, \beta}^{\varphi}$.
Next we compute the invariant $T_{M, \beta}^{\varphi}$ for a Seifert manifold $M$ over $S^2$ (Theorem \ref{thm_T}).
Note that every Seifert manifold which is a result of Dehn surgery along a knot has $S^2$ or $\R P^2$ as its base space.
Finally as an application we consider the Kinoshita-Terasaka knot $KT$, whose Alexander polynomial is $1$.
We show that for any integer $q$, $KT(6/q)$ is not homeomorphic to any Seifert manifold over $S^2$ with three singular fibers.
In this case we can check that abelian Reidemeister torsion gives no information.

This paper is organized as follows.
In the next section we give a brief exposition of fundamental facts about Reidemeister torsion.
In Section $3$ we develop a key lemma of Reidemeister torsion on gluing a solid torus along a torus boundary.
Furthermore we define the invariants $T_K^{\varphi}([g, h])$ and $T_{M, \beta}^{\varphi}$ and describe a Dehn surgery formula on these invariants.
Section $4$ is devoted to computations of $T_{M, \beta}^{\varphi}$ for Seifert manifolds over $S^2$.
In the last section we apply these results to the Kinoshita-Terasaka knot.

\section{Reidemeister torsion}
We first review the definition of Reidemeister torsion.
See \cite{M} and \cite{T} for more details.

For given bases $v$ and $w$ of a vector space, we denote by $[v / w]$ the determinant of the base change matrix from $w$ to $v$.

Let $\F$ be a commutative field and $C_* = (C_m \xrightarrow{\partial_m} C_{m-1} \to \cdots \to C_0)$ an acyclic chain complex of finite dimensional vector spaces over $\F$.
For a basis $b_i$ of $\im \partial_{i+1}$ for $i = 0, 1, \dots, m$, choosing a lift of $b_{i-1}$ in $C_i$ and combining it with $b_i$, we obtain a basis $b_i b_{i-1}$ of $C_i$.
\begin{defn}
For  a given basis $\boldsymbol{c} = \{ c_i \}$ of $C_*$, we choose a basis $\{ b_i \}$ of $\im \partial_*$ and define
\[ \tau(C_*, \boldsymbol{c}) := \prod_{i=0}^m [b_i b_{i-1} / c_i]^{(-1)^{i+1}} ~ \in \F^*. \]
\end{defn}
It can be easily checked that $\tau(C_*, \boldsymbol{c})$ does not depend on the choices of $b_i$ and $b_i b_{i-1}$.

The torsion $\tau(C_*, \boldsymbol{c})$ has the following multiplicative property.
Let
\[ 0 \to C_*' \to C_* \to C_*'' \to 0 \]
be a short exact sequence of acyclic chain complexes and $\boldsymbol{c} = \{ c_i \}$, $\boldsymbol{c}' = \{ c_i' \}$ and $\boldsymbol{c}'' = \{ c_i'' \}$ bases of $C_*$, $C_*'$ and $C_*''$ respectively.
Choosing a lift of $c_i''$ in $C_i$ and combining it with the image of $c_i'$ in $C_i$, we obtain a basis $c_i' c_i''$ of $C_i$.
\begin{thm}(\cite[Theorem 3.\ 1]{M}, \cite[Theorem 1.\ 5]{T}) \label{thm_M}
If $[c_i' c_i'' / c_i] = 1$ for all $i$, then
\[ \tau(C_*, \boldsymbol{c}) = \tau(C_*', \boldsymbol{c}') \tau(C_*'', \boldsymbol{c}''). \]
\end{thm}

Let $X$ be a connected finite CW-complex and $\rho \colon \pi_1 X \to GL_n(R)$ a linear representation over a commutative ring $R$.
We regard $R^n$ as a left $\Z[\pi_1 X]$-module by
\[ \gamma \cdot v :=  \rho(\gamma) v, \]
where $\gamma \in \pi_1 X$ and $v \in R^n$.
Then we define the twisted homology group and the twisted cohomology group of $X$ associated to $\rho$ as follows:
\begin{align*}
H_i^{\rho}(X; R^n) &:= H_i(C_*(\widetilde{X}) \otimes_{\Z[\pi_1 X]} R^n), \\
H_{\rho}^i(X; R^n) &:= H^i(\hom_{\Z[\pi_1 X]}(C_*(\widetilde{X}), R^n)),
\end{align*}
where $\widetilde{X}$ is the universal covering of $X$.
 
\begin{defn}
For a representation $\rho \colon \pi_1 X \to GL_n(\F)$ with $H_*^{\rho}(X; F^n) = 0$, we define the \textit{Reidemeister torsion} $\tau_{\rho}(X)$ of $X$ associated to $\rho$ as follows.
We choose a lift $\tilde{e}_i$ in $\widetilde{X}$ for each cell $e_i$ of $X$ and a basis $\langle f_1, \dots, f_n \rangle$ of $\F^n$.
Then
\[ \tau_{\rho}(X) := [\tau(C_*^{\rho}(X; \F^n), \tilde{\boldsymbol{c}})] ~ \in \F^* / (\pm 1)^n \im \det \circ \rho, \]
where
\[ \tilde{\boldsymbol{c}} := \langle \tilde{e}_1 \otimes f_1, \dots , \tilde{e}_1 \otimes f_n, \dots , \tilde{e}_{\dim C_*(X)} \otimes f_1, \dots , \tilde{e}_{\dim C_*(X)} \otimes f_n \rangle. \]
For a representation $\rho \colon \pi_1 X \to GL_n(\F)$ with $H_*^{\rho}(X; F^n) \neq 0$, we set $\tau_{\rho}(X) = 0$.
\end{defn}
It is known that $\tau_{\rho}(X)$ does not depend on the choices of $\tilde{e}_i$ and $\langle f_1, \dots, f_n \rangle$ and is a simple homotopy invariant.
\begin{rem} \label{rem_A}
For a link exterior of $S^3$, given a presentation of the link group, Reidemeister torsion can be computed efficiently using Fox calculus (cf.\ e.g. \cite{KL}, \cite{Ki1}).
\end{rem}

\section{A surgery formula}
\subsection{A gluing lemma}
In this subsection we discuss a gluing lemma (Proposition \ref{prop_G}) which we need to establish a surgery theorem (Theorem \ref{thm_S}) and to compute Reidemeister torsion of Seifert manifolds (Lemma \ref{lem_T}).

Let $E$ be a compact connected orientable 3-manifold whose boundary consists of tori and $M$ a 3-manifold obtained by gluing a solid torus $Z$ to $E$ along a component of $\partial E$.
We take a generator $\nu \in \pi_1 Z$ and a representation $\rho \colon \pi_1 M \to GL_n(\F)$.
Let us denote by $\pi$ and $i$ the homomorphisms $\pi_1 E \to \pi_1 M$ and $\pi_1 Z \to \pi_1 M$ induced by the inclusion maps respectively.
\begin{prop} \label{prop_G}
If there exists $\gamma \in \pi_1 M$ such that $\det(\rho(\gamma) - I) \neq 0$, then
\[ \tau_{\rho \circ \pi}(E) = [\det(\rho \circ i(\nu) - I)] \tau_{\rho}(M). \]
\end{prop}

To prove this proposition we begin by collecting the following computations.
\begin{lem} \label{lem_C}
(i) The following conditions are equivalent. \\
(a) $H_*^{\rho \circ i}(Z; \F^n)$ vanishes. \\
(b) $H_*^{\rho \circ i}(\partial Z; \F^n)$ vanishes. \\
(c) $\det(\rho \circ i(\nu) - I) \neq 0$. \\
(ii) If $\rho$ satisfies one of the conditions in (i), then
\begin{align*}
\tau_{\rho \circ i}(Z) &= [\det(\rho \circ i(\nu) - I)^{-1}], \\
\tau_{\rho \circ i}(\partial Z) &= [1].
\end{align*}
\end{lem}

\begin{proof}
We only consider the case of $\partial Z$.
The proof for the case of $Z$ is very similar.
Taking the natural cell structure on $\partial Z$ with one 0-cell, two 1-cells and one 2-cell, one can identify $C_*^{\rho \circ i}(\partial Z; \F^n)$ with
\[ 0 \to \F^n \xrightarrow{\partial_2} \F^{2n} \xrightarrow{\partial_1} F^n \to 0, \]
where
\[ \partial_1 = 
\begin{pmatrix}
\rho(\nu^{-1}) - I & \boldsymbol{0}
\end{pmatrix}
~\text{and}~
\partial_2 =
\begin{pmatrix}
\boldsymbol{0} \\
\rho(\nu^{-1}) - I
\end{pmatrix}.
\]
Therefore $H_*^{\rho \circ i}(\partial Z; \F^n)$ vanishes if and only if $\det(\rho \circ i(\nu) - I) \neq 0$ and for appropriate choices of bases $\{ b_i \}$ and $\langle f_1, \dots, f_n \rangle$,
\[ \tau_{\rho \circ i}(\partial Z) = \biggl[ \frac{\det(\rho \circ i(\nu^{-1}) - I)}{\det(\rho \circ i(\nu^{-1}) - I)} \biggr] = [1]. \]
\end{proof}

We define a representation $\rho^{\dagger}$ of $\pi_1 M$ to be
\[ \rho^{\dagger}(\gamma) := \rho(\gamma^{-1})^T, \]
where $\gamma \in \pi_1 M$.
Then we have an isomorphism
\begin{equation} \label{eq_D}
C_{\rho^{\dagger}}^*(M; \F^n) \cong \hom(C_*^{\rho}(M; \F^n), \F)
\end{equation}
defined by
\[ \psi \mapsto (c \otimes v \mapsto \psi(c)^T v), \]
where $\psi \in C_{\rho^{\dagger}}^*(M; \F^n)$, $c \in C_*(\widetilde{M})$ and $v \in \F^n$.

\begin{proof}[Proof of Proposition \ref{prop_G}]
We first prove that (a) $H_*^{\rho \circ \pi}(E; \F^n)$ vanishes if and only if (b) $H_*^{\rho}(M; \F^n)$ vanishes and (c) $\det(\rho \circ i(\nu) - I) \neq 0$.
By Lemma \ref{lem_C}(i) and the Mayer-Vietoris long exact sequence we check at once that two of the conditions (a), (b) and (c) deduce the other one.
Therefore it suffices to show that (a) deduce (c).

Let us assume that (a) holds and that $\det(\rho \circ i(\nu) - I) = 0$.
From the proof of Lemma \ref{lem_C} one can see that $H_2^{\rho \circ i}(\partial Z; \F^n) \neq 0$.
By the Mayer-Vietoris long exact sequence we obtain $H_3^{\rho}(M; \F^n) \neq 0$.
If $\partial M \neq \emptyset$, then $M$ collapses onto a 2-dimensional subcomplex, which contradicts it.
If $M$ is closed, then by Poincar\'{e} duality, \eqref{eq_D} and the universal coefficient theorem we have
\begin{align*}
H_0^{\rho^{\dagger}}(M; \F^n) &\cong H_{\rho^{\dagger}}^3(M; \F^n) \\
&\cong H^3(\hom(C_*^{\rho}(M; \F^n), \F)) \\
&\cong \hom(H_3^{\rho}(M; \F^n), \F) \qquad \neq 0.
\end{align*}
However, there exists $\gamma \in \pi_1 M$ such that $\det(\rho^{\dagger}(\gamma) - I) \neq 0$, and so $H_0^{\rho^{\dagger}}(M; \F^n) = 0$, a contradiction.

Next we assume that $H_*^{\rho \circ \pi}(E; \F^n)$ vanishes.
It follows from the above argument that $\tau_{\rho}(M)$ is defined.
By Lemma \ref{lem_C}(i) $\tau_{\rho \circ i}(Z)$ and $\tau_{\rho \circ i}(\partial Z)$ are also defined.
Considering the exact sequence
\[ 0 \to C_*^{\rho \circ i}(\partial Z; \F^n) \to C_*^{\rho \circ \pi}(E; \F^n) \oplus C_*^{\rho \circ i}(Z; \F^n) \to C_*^{\rho}(M; \F^n) \to 0, \]
by the multiplicative property of torsion (Theorem \ref{thm_M}) we obtain
\[ \tau_{\rho \circ \pi}(E) \tau_{\rho \circ i}(Z) = \tau_{\rho}(M) \tau_{\rho \circ i}(\partial Z). \]
Combining it with Lemma \ref{lem_C} (ii), we completes the proof.
\end{proof}

\subsection{Description of the formula}
Fix a finite group $G$.
For a group $\Pi$, we denote by $S(\Pi, G)$ the set of conjugacy classes of surjective homomorphisms from $\Pi$ to $G$.
Let $K$ be an oriented smooth knot in an oriented homology $3$-sphere.
We take a longitude-meridian pair $\lambda$, $\mu \in \pi_1 E_K$ which is compatible with the orientations of $K$ and the ambient space and define the abelianization map $\alpha \colon \pi_1 E_K \to \langle t \rangle$ which maps $\mu$ to $t$.

\begin{defn} \label{def_T}
Let $\varphi \colon G \to GL_n(\F)$ be a representation. \\
(i)For $[g, h] \in G \times G / G$, we define $T_K^{\varphi}([g, h])$ to be the set of $\tau_{\alpha \otimes (\varphi \circ \rho)}(E_K)$ for $[\rho] \in S(\pi_1 E_K, G)$ such that $[\rho(\lambda), \rho(\mu)] = [g, h]$, where $\alpha \otimes (\varphi \circ \rho)$ is a representation $\pi_1 E_K \to GL_n(\F(\zeta))$ which maps $\gamma \in \pi_1 E_K$ to $\alpha(\gamma) (\varphi \circ \rho)(\gamma)$. \\
(ii)For a closed connected 3-manifold $M$ with $H_1(M) = \Z / p$ and a surjection $\beta \colon \pi_1 M \to \langle \zeta \rangle$, where $\zeta \in \overline{\F}$ is a primitive $p$-root of $1$, we define $T_{M, \beta}^{\varphi}$ to be the set of $\tau_{\beta \otimes (\varphi \circ \rho)}(M)$ for  $[\rho] \in S(\pi_1 M, G)$, where $\beta \otimes (\varphi \circ \rho)$ is defined as $\alpha \otimes (\phi \circ \rho)$.
\end{defn}

\begin{thm} \label{thm_S}
We take integers $r$ and $s$ such that $p s - q r = 1$.
Let $\beta \colon \pi_1 K(p / q) \to \langle \zeta \rangle$ be a surjection which maps the image $[\mu]$ to $\zeta$.
If for any $[g, h]$ such that $g^q h^p = 1$ and $T_K^{\varphi}([g, h])$ is not empty, $\det(\zeta \varphi(h) - I) \neq 0$ and $\det(\zeta^r \varphi(g^s h^r) - I) \neq 0$, then
\[ T_{K(p / q), \beta}^{\varphi} = \biggl\{ \frac{\tau |_{t =\zeta}}{[\det(\zeta^r \varphi(g^s h^r) - I)]} ~;~ \tau \in T_K^{\varphi}([g, h]) \text{ with } g^q h^p = 1 \biggr\}. \]
\end{thm}

This theorem easily follows from Proposition \ref{prop_G} and the following lemma.
\begin{lem}
Let $\alpha' \colon \pi_1 E_K \to \langle \zeta \rangle$ be a surjection which maps $\mu$ to $\zeta$ and $\rho \colon \pi_1 E_K \to GL_n(\F)$ a representation.
If $\det(\zeta \rho(\mu) - I) \neq 0$, then 
\[ \tau_{\alpha' \otimes \rho}(E_K) = \tau_{\alpha \otimes \rho}(E_K)|_{t = \zeta}. \]
\end{lem}

\begin{proof}
Choose a triangulation of $E_K$ and maximal trees $T$ and $T'$ in the $1$-skeleton and in the dual $1$-skeleton respectively.
Collapsing $T$ and all the $3$-cells along $T'$, we have a $2$-dimensional CW-complex $W$ which is simple homotopic to $E_K$.
Let us denote the number of $1$-cells of $W$ by $m$, then it follows from $\chi(E_K) = 0$ that there are $(m-1)$ $2$-cells.
We can arrange the chain complex $C_*(\widetilde{W})$ of the form
\[ 0 \to C_2(\widetilde{W}) \xrightarrow{\partial_2} C_1(\widetilde{W}) \xrightarrow{\partial_1} C_0(\widetilde{W}) \to 0, \] 
where
\[ \partial_1 =
\begin{pmatrix}
\gamma_1 - 1 & \dots & \gamma_m - 1
\end{pmatrix}
\]
and $\{ \gamma_1, \dots, \gamma_m \}$ is a generator set of $\pi_1 W$.
If necessary, attaching one $1$-cell and one $2$-cell along the word of $\mu$ in $\gamma_1, \dots, \gamma_m$, we can assume that $\gamma_1 = \mu$.
Let $A$ be the result of deleting $1$st row of the matrix of $\partial_2$.

First we assume that $H_*^{\alpha' \otimes \rho}(E_K; \F(\zeta)^n)$ vanishes.
Then $H_2^{\alpha' \otimes \rho}(E_K; \F(\zeta)^n) = 0$ and $\det(\zeta \rho(\mu) - I) \neq 0$ deduce $\det(\alpha' \otimes \rho(A)) \neq 0$, and so $\det(\alpha \otimes \rho(A)) \neq 0$, where $\alpha' \otimes \rho(A)$ is the $(m-1) n$-dimensional matrix with entries in $\F(\zeta)$ which is the result that $\alpha' \otimes \rho$ linearly operates all the entries of $A$ and $\alpha \otimes \rho(A)$ is defined similarly.
This gives $H_2^{\alpha \otimes \rho}(E_K; \F(t)^n) = 0$.
Since $\det(t \rho(\mu) - I) \neq 0$, we obtain $H_0^{\alpha \otimes \rho}(E_K; \F(t)^n) = 0$.
Considering 
\[ \sum_{i = 0}^2 (-1)^i \dim H_i^{\alpha \otimes \rho}(E_K; \F(t)^n) = n \chi(E_K) = 0, \]
we can see that $H_*^{\alpha \otimes \rho}(E_K; \F(t)^n)$ vanishes.
In this case we have
\begin{align*}
\tau_{\alpha \otimes \rho}(E_K)|_{t = \zeta} &= \biggl[ \frac{\det(\alpha \otimes \rho(A))}{\det(t \rho(\mu) - I)} \bigg|_{t = \zeta} \biggr] \\
&= \biggl[ \frac{\det(\alpha' \otimes \rho(A))}{\det(\zeta \rho(\mu) - I)} \biggr] \\
&= \tau_{\alpha' \otimes \rho}(E_K)|_{t = \zeta} \quad \neq 0.
\end{align*}

Now assume that $H_*^{\alpha \otimes \rho}(E_K; \F(t)^n)$ vanishes and that $\tau_{\alpha \otimes \rho}(E_K)|_{t = \zeta} \neq 0$.
Then $\det(\alpha' \otimes \rho(A)) \neq 0$, and so the same argument as above shows that $H_*^{\alpha' \otimes \rho}(E_K; \F(t)^n)$ vanishes.
These prove the lemma.
\end{proof}

\section{Torsion of Seifert manifolds}

In this section we compute the invariant $T_{M, \beta}^{\varphi}$ for a Seifert manifold $M$ over $S^2$.

Let $L$ be the link in $S^3$ represented in Figure \ref{fig_S} and $E_L$ the exterior of an open tubular neighborhood of $L$.
We denote by $M(p_1 / q_1, p_2 / q_2, \dots, p_m / q_m)$ the 3-manifold which has a surgery description shown in Figure \ref{fig_S} and take integers $r_i$ and $s_i$ such that $p_i s_i - q_i r_i = 1$ for $i = 1, \dots, m$.
We assume that $m \geq 2$ and that $p_i \geq 2$ for $i = 1, \dots, m$.

\begin{figure}
\centering
\includegraphics[width=5cm, clip]{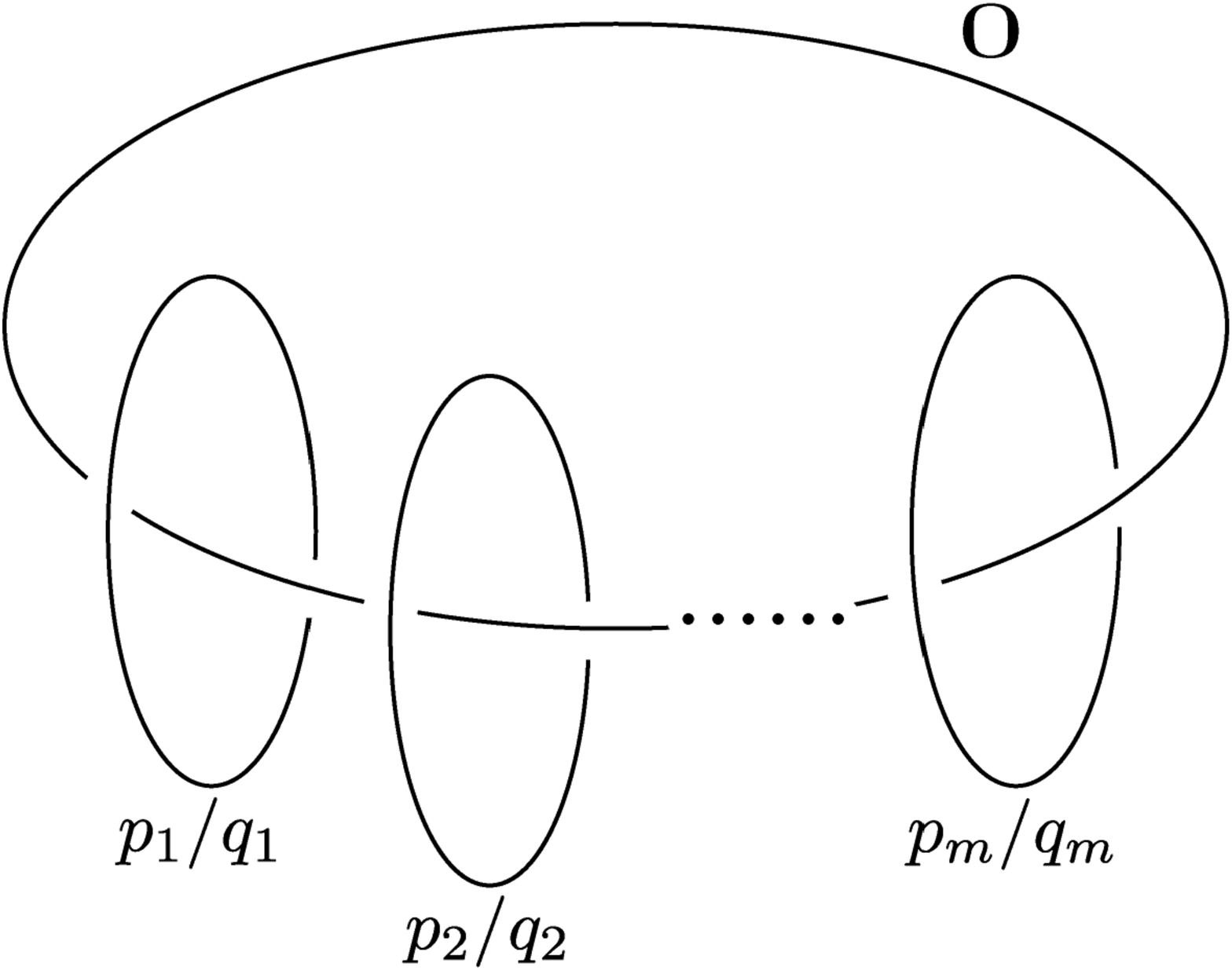}
\caption{The Seifert manifold $M(p_1 / q_1, p_2 / q_2, \dots, p_m / q_m)$}
\label{fig_S}
\end{figure}

From the diagram we have presentations of $\pi_1 E_L$ and $\pi_1 M(p_1 / q_1, \dots, p_m / q_m)$ as follows:
\begin{align}
\pi_1 E_L &= \langle x, y_1, y_2, \dots, y_m ~|~ [x, y_i] = 1 \text{ for } i = 1, \dots, m \rangle, \label{eq_L} \\
\pi_1 M(p_1 / q_1, \dots, p_m / q_m) &= \langle x, y_1, y_2, \dots, y_m ~|~ y_1 \dots y_m = 1, [x, y_i] = x^{q_i} y_i^{p_i} = 1 \text{ for } i = 1, \dots, m \rangle. \label{eq_S}
\end{align}

We fix a finite group $G$.
The group $G$ acts on $G^{m+1}$ by
\[ g' \cdot (g, h_1, \dots, h_m) := (g' g g'^{-1}, g' h_1 g'^{-1}, \dots, g' h_m g'^{-1}) \]
for $g' \in G$ and $(g, h_1, \dots, h_m) \in G^{m+1}$.
\begin{defn}
We define $S_G(p_1 / q_1, \dots, p_m / q_m)$ to be the set of $[g, h_1, \dots, h_m] \in G^{m+1} / G$ such that
\[ \langle g, h_1, \dots, h_m \rangle = G,~ g \in Z(G),~ h_1 \dots h_m = 1 ~\text{ and }~ g^{q_i} h_i^{p_i} = 1 \text{ for } i = 1, \dots, m, \]
where $Z(G)$ is the center of $G$.
\end{defn}

\begin{lem} \label{lem_S}
The map $S(\pi_1 M(p_1 / q_1, \dots, p_m / q_m), G) \to S_G(p_1 / q_1, \dots, p_m / q_m)$ which maps $[\rho]$ to $[\rho(x), \rho(y_1), \dots, \rho(y_m)]$ is bijective.
\end{lem}
The proof is straightforward from \eqref{eq_S}.

\begin{lem} \label{lem_T}
Let $\rho \colon \pi_1 M(p_1 / q_1, \dots, p_m / q_m) \to GL_n(\F)$ be a representation.
If $\det(\rho(x) - I) \neq 0$, then
\[ \tau_{\rho}(M(p_1 / q_1, \dots, p_m / q_m)) = \biggl[ \frac{\det(\rho(x) - I)^{m-2}}{\prod_i^m \det(\rho(x^{s_i} y_i^{r_i}) - I)} \biggr]. \]
\end{lem}

\begin{proof}
Let $\pi \colon \pi_1 E_L \to \pi_1 M(p_1 / q_1, \dots, p_m / q_m)$ be the natural surjection.
From \eqref{eq_L} we can directly compute that 
\[ \tau_{\rho \circ \pi}(E_L) = [\det(\rho(x) - I)^{m-2}] \]
(Remark \ref{rem_A}).
The details are left to the reader.
Now we use Proposition \ref{prop_G} repetitiously, and the lemma follows.
\end{proof}

Now we easily obtain the next theorem as a corollary of Lemma \ref{lem_S} and Lemma \ref{lem_T}. 
\begin{thm} \label{thm_T}
Let $\varphi \colon G \to GL_n(\F)$ be a representation and $\beta \colon \pi_1 M(p_1 / q_1, \dots, p_m / q_m) \to \langle \zeta \rangle$ a surjection, which maps $x$ to $\zeta^a$ and $y_i$ to $\zeta^{b_i}$ for $i = 1, \dots, m$.
If for any $[g, h_1, \dots. h_m] \in S_G(p_1 / q_1, \dots, p_m / q_m)$, $\det(\zeta^a \varphi(g) - I) \neq 0$, then
\[ T_{M(p_1 / q_1, \dots, p_m / q_m), \beta}^{\varphi} = \biggl\{ \biggl[ \frac{\det(\zeta^a \varphi(g) - I)^{m-2}}{\prod_{i = 1}^m \det(\zeta^{a s_i + b_i r_i} \varphi(g^{s_i} h_i^{r_i}) - I)} \biggr] ~;~ [g, h_1, \dots, h_m] \in S_G(p_1 / q_1, \dots, p_m / q_m) \biggr\}. \]
\end{thm}

\begin{rem}
In \cite{Ki2} Kitano gave a formula which computes $\tau_{\rho}(M)$ for a general Seifert manifold $M$ and an irreducible representation $\rho \colon \pi_1 M \to SL_n(\C)$ such that $H_*^{\rho}(M; \C^n)$ vanishes.
\end{rem}

\section{Application}

Let $KT$ be the Kinoshita-Terasaka knot illustrated in Figure \ref{fig_KT}.
It is well known that $\Delta_{KT} = 1$.
As an application we show that $KT(6/q)$ is not homeomorphic to $M(p_1/q_1, p_2/q_2, p_3/q_3)$ for any integer $q$ and any pair $(p_1/q_1, p_2/q_2, p_3/q_3)$.

\begin{figure}
\centering
\includegraphics[width=5cm, clip]{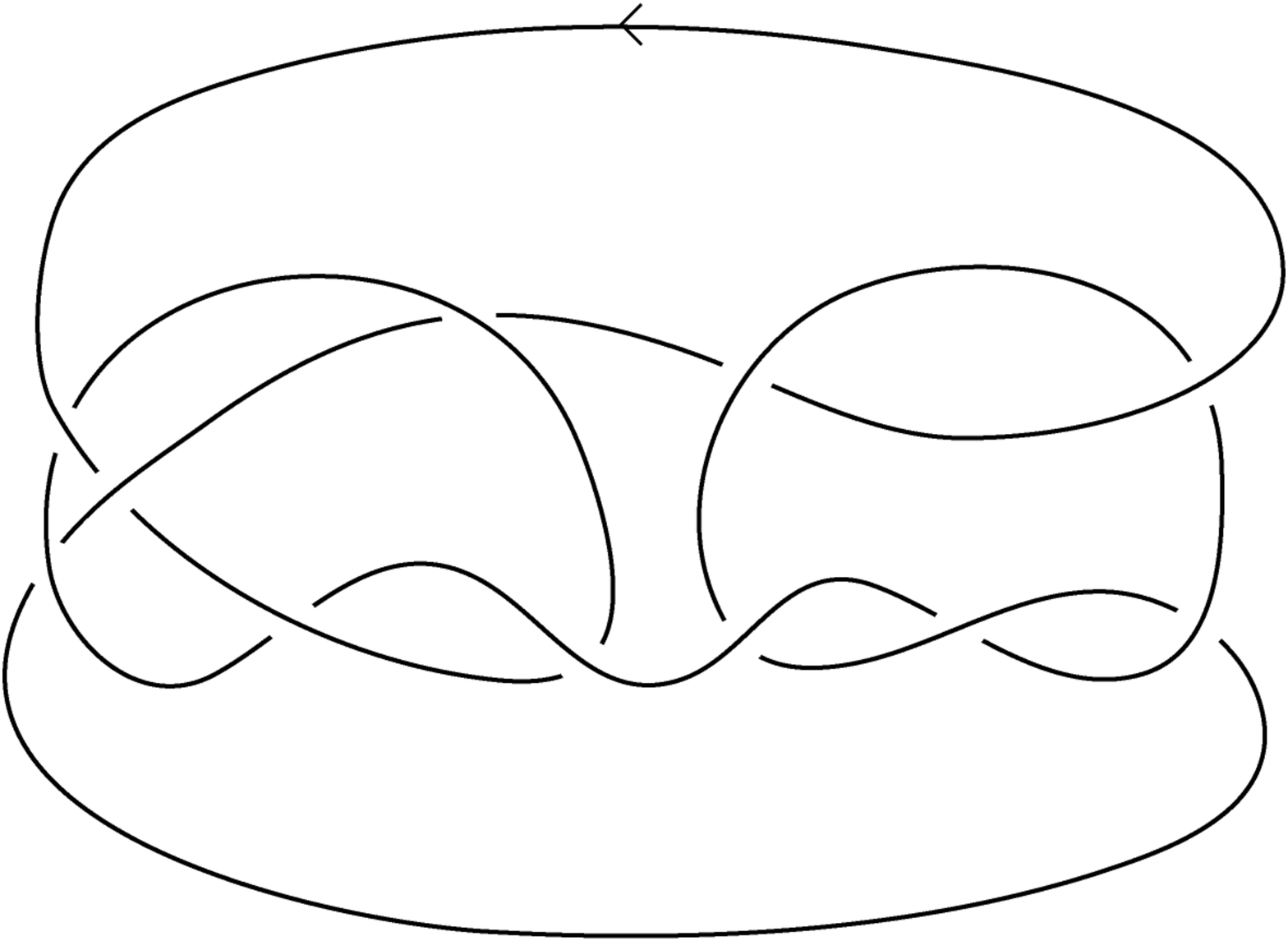}
\caption{The Kinoshita-Terasaka knot}
\label{fig_KT}
\end{figure}

For example, let us consider $M(3 / 2, -3, -5)$, whose $1$st homology group is $\Z/6$.
We set $\zeta = e^{\frac{\sqrt{-1} \pi}{3}}$.
Since we can compute that
\[ \tau_{\alpha'}(E_{KT}) = [1] \]
for any surjection $\alpha' \colon \pi_1 E_K \to \langle \zeta \rangle$ (Remark \ref{rem_A}), it follows from Proposition $\ref{prop_G}$ that
\[ \tau_{\beta}(KT(6 / q)) = [1] \]
for any surjection $\beta \colon \pi_1 KT(6 / q) \to \langle \zeta \rangle$.
Furthermore Lemma \ref{lem_T} yields
\[ \tau_{\beta'}(M(3 / 2, -3, -5)) = [1] \]
for any surjection $\beta' \colon \pi_1 M \to \langle \zeta \rangle$, hence abelian Reidemeister torsion gives no information in this case.

First we have the following data on $KT(6/q)$.
By direct computations we obtain
\begin{align}
S(\pi_1 KT(6/q), \mathfrak{A}_4) &= \emptyset, \label{eq_E1} \\
\sharp S(\pi_1 KT(6/q), \mathfrak{A}_5) &= 2, \label{eq_E2}
\end{align}
where $\mathfrak{A}_n$ is the alternating group on $n$ letters.
Let $\varphi \colon \mathfrak{A}_5 \to SL_4(\C)$ be the representation induced by the natural action of the symmetric group $\mathfrak{S}_5$ on $\C^5 / \C(1, 1, 1, 1, 1)$.
Then we computes that
\[ T_{KT}^{\varphi}([g, h]) = 
\begin{cases}
\{ [(t^2 + t + 1) (5 t^6 + 5 t^5 - 5 t^4 - 9 t^3 - 5 t^2 + 5 t + 5) (t - 1)^4] \} \\
\text{if } [g, h] = [1, (3, 4, 5)], \\
\emptyset \quad \text{otherwise}
\end{cases}
\]
(Remark \ref{rem_A}).
By Theorem \ref{thm_S} we have
\begin{equation} \label{eq_E3}
T_{KT(6 / q), \beta}^{\varphi} = \{ [29] \}
\end{equation}
for any surjection $\beta \colon \pi_1 KT(6 / q) \to \langle \zeta \rangle$. 

Second we have the following lemma on $M(p_1/q_1, p_2/q_2, p_3/q_3)$.
\begin{lem} \label{lem_A}
Let $\beta' \colon \pi_1 M(p_1/q_1, p_2/q_2, p_3/q_3) \to \langle \zeta \rangle$ be a surjection, which maps $x$ to $\zeta^a$.
If $6 \nmid a$, then for any $\tau \in T_{M, \beta'}^{\varphi}$,
\[ |\tau| = \frac{A}{B_1 B_2 B_3}, \]
where
\begin{align*}
A &= 1, 9, 16, \\
B_i &= 1, 2, 4, 9, 16 \quad \text{for } i = 1, 2, 3.
\end{align*}
\end{lem}

\begin{proof}
By Theorem \ref{thm_T} there exist $c_i \in \Z$ and $h_i' \in \mathfrak{A}_5$ for $i = 1, 2, 3$ such that
\[ \tau = \left[ \frac{(\zeta^a - 1)^4}{\prod_{i=1}^3 \det(\zeta^{c_i} \varphi(h_i') - I)} \right]. \]
Note that $Z(\mathfrak{A}_5) = 1$.
The possible values of $|(\zeta^a - 1)^4|$ are $1$, $9$, $16$ and these of $|\det(\zeta^{c_i} \varphi(h_i') - I)|$ are $1$, $2$, $4$, $9$, $16$, which proves the lemma.
\end{proof}

Now let us suppose that $KT(6/q)$ is homeomorphic to $M(p_1/q_1, p_2/q_2, p_3/q_3)$.
Since $H_1(M(p_1/q_1, p_2/q_2, p_3/q_3)) = \Z/6$ we have
\begin{equation} \label{eq_E4}
|q_1 p_2 p_3 + p_1 q_2 p_3 + p_1 p_2 q_3| = 6.
\end{equation}
From \eqref{eq_E1} and \eqref{eq_E2} we have
\begin{align*}
S_{\mathfrak{A}_4}(p_1/q_1, p_2/q_2, p_3/q_3) &= \emptyset, \\
\sharp S_{\mathfrak{A}_5}(p_1/q_1, p_2/q_2, p_3/q_3) &= 2.
\end{align*}
By direct computations these are equivalent to the conditions that ($0$) we cannot realize that
\[ 2 \mid p_1, \quad 3 \mid p_2, \quad 3 \mid p_3 \]
by permuting the indices and that only one of the following holds:
\begin{align*}
\text{(i) after possible permuting the indices,} \quad  2 \mid p_1, \quad 3 \mid p_2, \quad 5 \mid p_3, \\
\text{(ii) after possible permuting the indices,} \quad 2 \mid p_1, \quad 5 \mid p_2, \quad 5 \mid p_3, \\
\text{(iii) after possible permuting the indices,} \quad 3 \mid p_1, \quad 3 \mid p_2, \quad 5 \mid p_3, \\
\text{(iv) after possible permuting the indices,} \quad 5 \mid p_1, \quad 5 \mid p_2, \quad 5 \mid p_3. \\
\end{align*}

In the case (i) we have $3 \mid p_1 p_3$ from \eqref{eq_E4}.
If $3 \mid p_1$, then (iii) also holds.
If $3 \mid p_3$, then (0) does not hold.
In the case (ii) we have $5 \mid p_1$ from \eqref{eq_E4}, and (iv) also holds.
In the case (iv) \eqref{eq_E4} does not hold.
Therefore we only have to consider the case (iii).

Let us assume (iii).
If $2 \mid p_1 p_2$, then (i) also holds, hence $2 \nmid p_1, p_2$.
Since
\[ \zeta^{a q_1 + b_1 p_1} = \zeta^{a q_2 + b_1 p_2} = \zeta^{b_1 + b_2 + b_3} = 1, \]
where $b_i$ is an integer such that $\beta'(y_i) = \zeta^{b_i}$ for $i = 1, 2, 3$,
if $2 \mid a$, then $2 \mid b_i$ for all $i$, and $\beta'$ cannot be surjective.
Therefore $2 \nmid a$ and , in consequence, the assumption of Lemma \ref{lem_A} is satisfied.
Comparing \eqref{eq_E3} and Lemma \ref{lem_A}, we have a contradiction, and we obtain the desired conclusion. \\

\noindent
\textbf{Acknowledgement.}
The author would like to express his gratitude to Toshitake Kohno for his encouragement and helpful suggestions.
He also would like to thank Hiroshi Goda, Teruhisa Kadokami, Takayuki Morifuji, Masakazu Teragaito and Yuichi Yamada for fruitful discussions and advices.
This research is supported by JSPS Research Fellowships for Young Scientists.


\end{document}